%% file: edwards-hales-raya.tex
\documentclass{llncs}

\usepackage{graphicx}
\usepackage{pdfpages}
\usepackage{xcolor}
\usepackage{amsfonts}
\usepackage{amsmath}
\usepackage{amscd}
\usepackage{amssymb}
\usepackage{alltt}
\usepackage{url}
\usepackage{ellipsis}

\usepackage{isabelle}
\usepackage{isabellesym}
\usepackage{isabelletags}
\usepackage{comment}
\usepackage{pdfsetup}
\usepackage{railsetup}
\usepackage{wasysym}

\usepackage{tikz} %
\usetikzlibrary{chains,shapes,arrows,%
 trees,matrix,positioning,decorations}

\newcommand{\ring}[1]{\mathbb{#1}}
\newcommand{\op}[1]{\hbox{#1}}
\newcommand{\f}[1]{\frac{1}{#1}}
\newcommand{\Eaff}{E_{\op{\scriptsize\it aff}}}
\newcommand{\Eaf}[1]{E_{\op{\scriptsize\it aff},{#1}}}
\newcommand{\Eoo}{E^{\circ}} 

\newcommand{\Go}{\langle\rho\rangle}

\newcommand{\ang}[1]{\langle{#1}\rangle}
\def\error{e__r__r__o__r}
\def\cong{\error}

\title{Formal Proof of the Group Law for Edwards Elliptic Curves}
\author{Thomas Hales\inst{1}\and Rodrigo Raya\inst{2}}
\date{}
\institute{University of Pittsburgh\and Technical University of Munich}

\begin{document}

\maketitle

\begin{abstract} 
  This article gives an elementary computational proof of the group
  law for Edwards elliptic curves. The associative law is expressed as
  a polynomial identity over the integers that is directly checked by
  polynomial division.  Unlike other proofs, no preliminaries such as
  intersection numbers, B\'ezout's theorem, projective geometry,
  divisors, or Riemann Roch are required.  The proof of the group law
  has been formalized in the Isabelle/HOL proof assistant.
\end{abstract}


\newenvironment{blockquote}{%
  \par%
  \medskip%
  \baselineskip=0.7\baselineskip%
  \leftskip=2em\rightskip=2em%
  \noindent\ignorespaces}{%
  \par\medskip}


\section{Introduction}

Elliptic curve cryptography is a cornerstone of mathematical
cryptography.  Many cryptographic algorithms (such as the
Diffie-Hellman key exchange algorithm which inaugurated public key
cryptography) were first developed in the context of the arithmetic of
finite fields.  The preponderance of finite-field cryptographic
algorithms have now been translated to an elliptic curve counterpart.
Elliptic curve algorithms encompass many of the fundamental
cryptographic primitives: pseudo-random number generation, digital
signatures, integer factorization algorithms, and public key exchange.

One advantage of elliptic curve cryptography over finite-field
cryptography is that elliptic curve algorithms typically obtain the same
level of security with smaller keys than finite-field algorithms.
This often means more efficient algorithms.

Elliptic curve cryptography is the subject of major international
cryptographic standards (such as NIST).  Elliptic curve cryptography
has been implemented in widely distributed software such as NaCl
\cite{bernstein2012security}.  Elliptic curve algorithms appear in
nearly ubiquitous software applications such as web browsers and
digital currencies.

The same elliptic curve can be presented in different ways by
polynomial equations.  The different presentations are known variously
as the Weierstrass curve $(y^2 = \text{cubic in } x)$, Jacobi curve
$(y^2 = \text{quartic in } x)$, and Edwards curve (discussed below).

The set of points on an elliptic curve forms an abelian group.  Explicit
formulas for addition are given in detail below.  The Weierstrass
curve is the most familiar presentation of an elliptic curve, but it
suffers from the shortcoming that the group law is not given by a
uniform formula on all inputs.  For example, special treatment must be
given to the point at infinity and to point doubling: $P
\mapsto 2P$.  Exceptional cases are bad; they are the source of 
hazards such as side-channel attacks (timing attacks) by
adversaries and implementation bugs \cite{brier2002weierstrass}.

Edwards curves have been widely promoted for cryptographic algorithms
because their addition law avoids exceptional cases and their hazards.
Every elliptic curve (in characteristic different from $2$) is
isomorphic to an elliptic curve in Edwards form (possibly after
passing to a quadratic extension).  Thus, there is little loss of
generality in considering elliptic curves in Edwards form.  For most
cryptographic applications, Edwards curves suffice.

The original contributions of this article are both mathematical and
formal.  Our proof that elliptic curve addition satisfies the axioms
of an abelian group is new (but see the literature survey below for
prior work).  Our proofs were designed with formalization specifically
in mind.  To our knowledge, our proof of associativity in Section
\ref{sec:assoc} is the most elementary proof that exists anywhere in
the published literature (in a large mathematical literature on
elliptic curves extending back to Euler's work on elliptic integrals).
Our proof avoids the usual machinery found in proofs of associativity
(such as intersection numbers, B\'ezout's theorem, projective
geometry, divisors, or Riemann Roch).  Our algebraic manipulations
require little more than multivariate polynomial division with
remainders, even avoiding Gr\"obner bases in most places.  Based on
this elementary proof, we give a formal proof in the Isabelle/HOL
proof assistant that every Edwards elliptic curve (in characteristic
other than $2$) satisfies the axioms of an abelian group.%
\footnote{Mathematica calculations are available at\\
  \url{https://github.com/thalesant/publications-of-thomas-hales/tree/master/cryptography/group_law_edwards}.\\
  The Isabelle/HOL formalization is available at\\
  \url{https://github.com/rjraya/Isabelle/blob/master/curves/Hales.thy}.}

It is natural to ask whether the proof of the associative law also
avoids exceptional cases (encountered in Weierstrass curves)
when expressed in terms of Edwards curves.
Indeed, this article gives a two-line proof of the associative law for
so-called \emph{complete} Edwards curves that avoids case splits and all
the usual machinery.

By bringing significant simplification to the fundamental proofs in
cryptography, our paper opens the way for the formalization of
elliptic curve cryptography in many proof assistants.  Because of its
extreme simplicity, we hope that our approach might be widely replicated
and translated into many different proof assistants.

\section{Published Literature}

A number of our calculations are reworkings of calculations found in
Edwards, Bernstein, Lange et al.~\cite{edwards2007normal},
\cite{bernstein2008twisted}, \cite{bernstein2007faster}.  A geometric
interpretation of addition for Edwards elliptic curves appears in
\cite{arene2011faster}.

Working with the Weierstrass form of the curve, Friedl was the first
to give a proof of the associative law of elliptic curves in a
computer algebra system (in Cocoa using Gr\"obner bases)
\cite{friedl}, \cite{friedl2017elementary}.  He writes, ``The
verification of some identities took several hours on a modern
computer; this proof could not have been carried out before the
1980s.''  These identities were later formalized in Coq with runtime one
minute and 20 seconds \cite{thery2007proving}.  A non-computational
Coq formalization based on the Picard group appears in
\cite{bartzia2014formal}.  By shifting to Edwards curves, we have
eliminated case splits and significantly improved the speed of the
computational proof.

An earlier unpublished note contains more detailed motivation,
geometric interpretation, pedagogical notes, and expanded proofs
\cite{hales2016group}.  The earlier version does not include
formalization in Isabelle/HOL.  Our formalization uncovered and
corrected some errors in the ideal membership problems in
\cite{hales2016group} (reaffirming the pervasive conclusion that
formalization catches errors that mathematicians miss).

Other formalizations of elliptic curve cryptography are found in Coq
and ACL2 by different methods \cite{russinoff2017computationally}.
After we posted our work to the arXiv, another formalization was given
in Coq along our same idea \cite{erbsen2017crafting}
\cite{erbsen2017systematic}.  It goes further by including
formalization of implementation of code, but it falls 
short of our work by not including the far more challenging and
interesting case of projective curves.

We do not attempt to survey the various formalizations of
cryptographic algorithms built on top of elliptic curves.  Because of
the critical importance of cryptography to the security industry, the
formalization of cryptographic algorithms is rightfully a priority
within the formalization community.

\section{Group Axioms}\label{sec:axiom}

This section gives an elementary proof of the group axioms for
addition on Edwards curves (Theorem~\ref{thm:group}).  We include
proofs, because our approach is not previously published.  

Our definition of Edwards curve is more inclusive than definitions
stated elsewhere. Most writers prefer to restrict to curves of genus
one and generally call a curve with $c\ne 1$ a twisted Edwards
curve. We have interchanged the $x$ and $y$ coordinates on the Edwards
curve to make it consistent with the group law on the circle.

\subsection{rings and homomorphisms}

In this section, we work algebraically over an arbitrary field $k$.
We assume a basic background in abstract algebra at the level of
a first course (rings, fields, homomorphisms, and kernels).  We set
things up in a way that all of the main identities to be proved are
identities of polynomials with integer coefficients.

All rings are assumed to be commutatative with identity $1\ne0$.
If $R$ is an integral domain
and if $\delta\in R$, then we write $R[\f{\delta}]$ for
the localization of $R$ with respect to the multiplicative set
$S=\{1,\delta,\delta^2,\ldots\}$; that is, the set of 
fractions with numerators in $R$ and denominators in $S$.  We
will need the well-known fact that if $\phi:R\to A$ is a ring
homomorphism sending $\delta$ to a unit in $A$, then $\phi$ extends
uniquely to a map $R[\f{\delta}]\to A$ that maps a fraction
$r/\delta^i$ to $\phi(r)\phi(\delta^i)^{-1}$.

\begin{lemma}[kernel property] Suppose that an identity $r = r_1 e_1 +
  r_2 e_2 +\cdots + r_k e_k$ holds in a commutative ring $R$.  If $\phi:R\to A$ is
  a ring homomorphism such that $\phi(e_i) =0$ for all $i$, then
  $\phi(r)=0$.
\end{lemma}

\begin{proof}
$\phi(r) = \sum_{i=1}^k \phi(r_i) \phi(e_i) = 0.$
\qed\end{proof}

We use the following rings: $R_0 := \ring{Z}[c,d]$ and $R_n :=
R_0[x_1,y_1,\ldots,x_n,y_n]$.  We introduce the polynomial for the
Edwards curve.  Let
\begin{equation}
e(x,y) = x^2 + c y^2 -1 - d x^2 y^2 \in  R_0[x,y].
\end{equation}

We write $e_i = e(x_i,y_i)$ for the image of the polynomial in $R_j$,
for $i\le j$, under $x\mapsto x_i$ and $y\mapsto y_i$.  Set
$\delta_x = \delta^-$ and $\delta_y = \delta^+$, where
\[\delta^{\pm} (x_1,y_1,x_2,y_2) = 1\pm d x_1 y_1 x_2 y_2\quad\text{and}\] 
\[
\delta(x_1,y_1,x_2,y_2) = \delta_x\delta_y\in R_2.
\]
We write $\delta_{ij}$ for its image of $\delta$ under
$(x_1,y_1,x_2,y_2)\mapsto (x_i,y_i,x_j,y_j)$.  So,
$\delta=\delta_{12}$.

\subsection{inverse and closure}

We write $z_i = (x_i,y_i)$.
We define a pair of rational functions that we denote using
the symbol $\oplus_0$:
\begin{equation}\label{eqn:add}
z_1 \oplus_0 z_2 =  \left(\frac{x_1 x_2 - c y_1 y_2}{1 - d x_1 x_2 y_1 y_2},
\frac{x_1 y_2 + y_1 x_2}{1+d x_1 x_2 y_1 y_2}\right) 
\in R_2[\f{\delta}]\times R_2[\f{\delta}].
\end{equation}
When specialized to $c=1$ and $d=0$, the polynomial $e(x,y)=x^2+y^2-1$ reduces to
a circle, and (\ref{eqn:add}) reduces to the standard group
law on a circle.
Commutativity is a consequence of the subscript symmetry
$1\leftrightarrow 2$ evident in the pair of rational functions:
\[
z_1 \oplus_0 z_2 = z_2\oplus_0 z_1.
\]
If $\phi:R_2[\f{\delta}]\to A$ is a ring homomorphism, we also write
$P_1\oplus_0 P_2\in A^2$ for the image of $z_1\oplus_0 z_2$.  We write
$e(P_i)\in A$ for the image of $e_i=e(z_i)$ under $\phi$.  We often
mark the image $\bar r=\phi(r)$ of an element with a bar accent.

Let $\iota(z_i) =\iota(x_i,y_i) = (x_i,-y_i)$.  The involution $z_i\to
\iota(z_i)$ gives us an inverse with properties developed below.

There is an obvious identity element $(1,0)$, expressed as follows.
Under a homomorphism $\phi:R_2[\f{\delta}]\to A$, mapping $z_1\mapsto
P$ and $z_2\mapsto (1,0)$, we have
\begin{equation}
P\oplus_0(1,0) = P.
\end{equation}

\begin{lemma} [inverse] 
  Let $\phi:R_2[\f{\delta}]\to A$, with $z_1\mapsto P$, $z_2\mapsto
  \iota(P)$.  If $e(P)=0$, then $P\oplus_0 \iota(P) = (1,0)$.
\end{lemma}

\begin{proof} Plug $P=(a,b)$ and $\iota\,P=(a,-b)$ into
  (\ref{eqn:add}) and use $e(P)=0$.
\qed\end{proof}

\begin{lemma}[closure under addition]\label{lemma:closure}
  Let $\phi:R_2[\f{\delta}]\to A$ with $z_i\mapsto P_i$.  If
  $e(P_1)=e(P_2)=0$, then
  \[
  e(P_1 \oplus_0 P_2) = 0.
  \]
\end{lemma}

\begin{proof} This proof serves as a model for several proofs that are
  based on multivariate polynomial division.  We write
\[
e(z_1\oplus_0 z_2) = \frac{r}{\delta^2},
\]
for some polynomial $r \in R_2$.  It is enough to show that
$\phi(r)=0$.  Polynomial division gives
\begin{equation}\label{eqn:closure}
r= r_1 e_1 + r_2 e_2,
\end{equation}
for some polynomials $r_i\in R_2$.  Concretely, the polynomials $r_i$
are obtained as the output of the one-line Mathematica command
\[
\op{PolynomialReduce}[r,\{e_1,e_2\},\{x_1,x_2,y_1,y_2\}].
\]
The result now follows from the kernel property and
(\ref{eqn:closure}); $ e(P_1) = e(P_2) = 0$ implies $\phi(r)= 0$,
giving ${e}(P_1\oplus_0 P_2)=0$.
\qed\end{proof}

Mathematica's {\tt PolynomialReduce} is an implementation of a naive
multivariate division algorithm \cite{cox1992ideals}.  In particular,
our approach does not require the use of Gr\"obner bases until
Section~\ref{sec:dichot}.  We write
\[
r \equiv r' \mod S,
\]
where $r-r'$ is a rational function and $S$ is a set of polynomials,
to indicate that the numerator of $r-r'$ has zero remainder when
reduced by polynomial division with respect to $S$ using {\tt
  PolynomialReduce}.  We also require the denominator of $r-r'$ to be
invertible in the localized polynomial ring.  The zero remainder will
give $\phi(r)=\phi(r')$ in each application.  We extend the notation
to $n$-tuples
\[
(r_1,\ldots,r_n) \equiv (r_1',\ldots,r_n') \mod S,
\]
to mean $r_i \equiv r_i' \mod S$ for each $i$.  Using this approach,
most of the proofs in this article almost write themselves.

\subsection{associativity}\label{sec:assoc}

This next step (associativity) is generally considered the hardest
part of the verification of the group law on curves.  Our proof is two
lines and requires little more than polynomial division.  The
polynomials $\delta_x,\delta_y$ appear as denominators in the addition
rule.  The polynomial denominators $\Delta_x,\Delta_y$ that appear
when we add twice are more involved.  Specifically, let $
(x_3',y_3')=(x_1,y_1) \oplus_0 (x_2,y_2)$, let $(x_1',y_1')=(x_2,y_2)
\oplus_0 (x_3,y_3) $, and set
\[
\Delta_x = \delta_x(x_3',y_3',x_3,y_3)
\delta_x(x_1,y_1,x_1',y_1')\delta_{12}\delta_{23}\in R_3.
\]
Define $\Delta_y$ analogously.

\begin{lemma}[generic associativity] \label{lemma:assoc} Let
  $\phi:R_3[\f{\Delta_x\Delta_y}]\to A$ be a homomorphism with
  $z_i\mapsto P_i$.  If $e(P_1)=e(P_2)=e(P_3)=0$, then
\[
(P_1 \oplus_0 P_2)\oplus_0 P_3 = 
P_1 \oplus_0 (P_2\oplus_0 P_3).
\]
\end{lemma}

\begin{proof} By polynomial division in the
  ring $R_3[\f{\Delta_x\Delta_y}]$
\[
((x_1,y_1)\oplus_0 (x_2,y_2)) \oplus_0 (x_3,y_3)\equiv
(x_1,y_1)\oplus_0 ((x_2,y_2) \oplus_0 (x_3,y_3)) \mod \{e_1,e_2,e_3\}.
\]
\qed\end{proof}

\subsection{group law for affine curves}

\begin{lemma}[affine closure] \label{lemma:affine} Let $\phi:R_2\to k$
  be a homomorphism into a field $k$.  If
  $\phi(\delta)=e(P_1)=e(P_2)=0$, then either $\bar d$ or $\bar c \bar
  d$ is a nonzero square in $k$.
\end{lemma}

The lemma is sometimes called completeness, in conflict with the usual
definition of \emph{complete} varieties in algebraic geometry.  To
avoid possible confusion, we avoid this terminology.  We use the lemma
in contrapositive form to give conditions on $\bar d$ and $\bar c\bar
d$ that imply $\phi(\delta)\ne0$.

\begin{proof} 
  Let $r = (1 - c d y_1^2 y_2 ^2) (1 - d y_1^2 x_2^2)$.  We have
\begin{equation}\label{eqn:squares}
r = d^2 y_1^2 y_2^2 x_2^2 e_1 + (1 - d y_1^2) \delta - d y_1^2 e_2.
\end{equation}
This forces $\phi(r)=0$, which by the form of $r$ implies that $\bar
c\bar d$ or $\bar d$ is a nonzero square.
\qed\end{proof}

We are ready to state and prove one of the main results of this
article.  This group law is expressed generally enough
to include the group law on the circle and ellipse as a special case
$\bar d = 0$.

\begin{theorem}[group law]\label{thm:group} 
  Let $k$ be a field, let $\bar c \in k$ be a square, and let $\bar
  d\not\in k^{\times 2}$.  
  Then 
  \[
  C= \{P\in k^2 \mid  e(P) = 0\}
  \]
   is an abelian
  group with binary operation $\oplus_0$.
\end{theorem}

\begin{proof} This follows directly from the earlier results.  For
  example, to check associativity of $P_1\oplus_0 P_2\oplus_0 P_3$, where
  $P_i\in C$, we define a homomorphism $\phi:R_3\to k$ sending
  $z_i\mapsto P_i$ and $(c,d)\mapsto (\bar c,\bar d)$.  By a repeated
  use of the affine closure lemma, $\phi(\Delta_y\Delta_x)$ is nonzero
  and invertible in the field $k$.  The universal property of
  localization extends $\phi$ to a homomorphism
  $\phi:R_3[\f{\Delta_y\Delta_x}]\to k$.  By the associativity lemma
  applied to $\phi$, we obtain the associativity for these three
  (arbitrary) elements of $C$.  The other group axioms follow similarly
  from the lemmas on closure, inverse, and affine closure.
\qed\end{proof}

The Mathematica calculations in this section are fast. For example,
the associativity certificate takes about $0.12$ second to compute on
a 2.13 GHz processor.  

\section{Formalization in Isabelle/HOL}

In this section, we describe the proof implementation in Isabelle/HOL.
We have formalized the two main theorems (Theorem~\ref{thm:group} and
Theorem~\ref{thm:proj-group}).  Formalization uses two different
locales: one for the affine and one for the projective case. (The
projective case will be discussed in Section~\ref{sec:proj}.)

Let $k$ be the underlying curve field. $k$ is introduced as the type
class \textit{field} with the assumption that $2 \neq 0$
(characteristic different from $2$). This is not included in the
simplification set, but used when needed during the proof.  The
formalized theorem is slightly less general than then informal
statement, because of this restriction.

\subsection{affine Edwards curves}

The formal proof fixes the curve parameters $c,d \in k$ (dropping the
bar accents from notation). The group addition $\oplus_0$ (of
Equation~\ref{eqn:add}) can be written as in Figure~\ref{fig:1}.  In
Isabelle's division ring theory, the result of division by zero is
defined as zero. This has no impact on validity of final results, but
gives cleaner simplifications in some proofs.

\begin{figure}
	{\input{proj-add-0.tex}}
	\caption{Definition of $\oplus_0$ in Isabelle/HOL}
	\label{fig:1}
\end{figure}

Most of the proofs in this section are straight-forward. The only
difficulty was to combine the Mathematica certificates of computation, 
into a single process in Isabelle.

In Figure~\ref{fig:2}, we show an excerpt of the proof of
associativity. We use the following abbreviations: 
\[
e_i = x_i^2 + c *
y_i^2 - 1 - d * x_i^2 * y_i^2 
\] 
where $e_i = 0$, since the involved
points lie on the curve and 
\[
\text{gxpoly} = ((p_1 \oplus_0 p_2)
\oplus_0 p_3 - p_1 \oplus_0 (p_2 \oplus p_3))_1*\Delta_x
\] 
which stands for a normalized version of the associativity law after
clearing denominators. We say that points are \emph{summable}, if the
rational functions defining their sum have nonzero denominators.
Since the points $p_i$ are assumed to be summable, $\Delta_x \neq
0$. As a consequence, the property stated in Figure~\ref{fig:2}
immediately implies that associativity holds in the first component of
the addition.

\begin{figure}
	{\input{proj-add-4.tex}}
	\caption{An excerpt of the proof of associativity}
	\label{fig:2}
\end{figure}

Briefly, the proof unfolds the relevant definitions and then
normalizes to clear denominators. The remaining terms of $\Delta_x$
are then distributed over addends. The unfolding and normalization of
addends is repeated in the lemmas \textit{simp1gx} and
\textit{simp2gx}. Finally, the resulting polynomial identity is proved
using the \textit{algebra} method.  Note that no computation was
required from an external tool.

The \emph{rewrite} tactic, which can modify a goal with various
rewrite rules in various locations (specified with a pattern), is
used to normalized terms \cite{noschinskipattern}.  Rewriting
in the denominators is sufficient for our needs.

For proving the resulting polynomial expression, the \textit{algebra}
proof method is used \cite{chaieb2007context}
\cite{chaieb2008automated} \cite{wenzel2019isabelle}.  Given
$e_i(x),\ p_{ij}(x),\ a_i(x) \in R[x_1,\ldots,x_n]$, where $R$ is a
commutative ring and $x=(x_1,\ldots,x_n)$, the method verifies
formulas
\[
\; \forall x.\ \bigwedge_{i = 1}^L
e_i({x}) = 0 \to \exists{y}.\ \bigwedge_{i = 1}^M
\left(a_i(x) = \sum_{j = 1}^N p_{ij}({x}) y_j \right)
\] 
The method is complete for such formulas that hold over all
commutative rings with unit \cite{harrison2007automating}.

\section{Group law for projective Edwards curves}\label{sec:proj}

By proving the group laws for a large class of elliptic curves,
Theorem \ref{thm:group} is sufficiently general for many applications
to cryptography.  Nevertheless, to achieve full generality, we
push forward.

This section shows how to remove the restriction $\bar d\not\in
k^{\times 2}$ that appears in the group law in the previous section.
By removing this restriction, we obtain a new proof of the group law
for all elliptic curves in characteristics different from $2$.
Unfortunately, in this section, some case-by-case arguments are
needed, but no hard cases are hidden from the reader.  The level of
exposition here is less elementary than in the previous section.
Again, we include proofs, because our approach is designed with
formalization in mind and has not been previously published.

The basic idea of our construction is that the projective
curve $E$ is obtained by gluing two affine curves $\Eaff$
together.  The associative property for $E$ is a consequence
of the associative property on affine pieces $\Eaff$, which
can be expressed as polynomial identities.

\subsection{definitions}\label{sec:defs}

In this section, we assume that $c\ne 0$ and that $c$ and $d$ are both
squares.  Let $t^2 = d/c$.  By a change of variable $y\mapsto
y/\sqrt{c}$, the Edwards curve takes the form
\begin{equation}\label{eqn:t}
e(x,y)= x^2 + y^2 -1 - t^2 x^2 y^2.
\end{equation}

We assume $t^2\ne 1$.  Note if $t^2=1$, then 
the curve degenerates to a product of intersecting lines, which
cannot be a group.  We also assume that $t\ne 0$, which only excludes the
circle, which has already been fully treated.  Shifting notation for
this new setting, let
\[
R_0 = \ring{Z}[t,\frac{1}{t^2-1},\frac1t],\quad
R_n = R_0[x_1,y_1,\ldots,x_n,y_n].
\]
As before, we write $e_i = e(z_i)$, $z_i=(x_i,y_i)$, and $ e(P_i) =
\phi(e_i)$ when a homomorphism $\phi$ is given.

Define rotation by $\rho(x,y)=(-y,x)$ and inversion $\tau$ by
\[
\tau(x,y) = (1/(tx),1/(ty)).
\]
Let $G$ be the abelian group of order eight generated by $\rho$ and
$\tau$.

\subsection{extended addition}

We extend the binary operation $\oplus_0$ using the automorphism $\tau$.
We also write $\delta_0$ for
$\delta$, $\nu_0$ for $\nu$ and so forth.

Set
\begin{equation}\label{eqn:tauplus}
  z_1\oplus_1 z_2 := \tau((\tau z_1)\oplus_0 z_2)=
  \left(\frac{x_1y_1 - x_2 y_2}{x_2
    y_1-x_1 y_2},\frac{x_1 y_1 + x_2 y_2}{x_1 x_2 + y_1 y_2}\right) 
= (\frac{\nu_{1x}}{\delta_{1x}},\frac{\nu_{1y}}{\delta_{1y}})
\end{equation}
in $R_2[\f{\delta_1}]^2$ where $\delta_1 = \delta_{1x}\delta_{1y}$.

We have the following easy identities of rational functions that are
proved by simplification of rational functions:
\begin{align}\label{eqn:r-tau}
\text{\it inversion invariance:}\quad\quad
\tau (z_1)\oplus_i z_2 &= z_1 \oplus_i \tau z_2;
\end{align}
\begin{align}\label{eqn:r-rho}
\text{\it rotation invariance:}\quad\quad
\begin{split}
\rho(z_1)\oplus_i z_2 &= \rho(z_1\oplus_i z_2);\\
\delta_i(z_1,\rho z_2) &= \pm \delta_i(z_1,z_2);
\end{split}
\end{align}
\begin{align}\label{eqn:r-iota}%
\text{\it inverses for $\sigma=\tau,\rho$:}\quad
\begin{split}
\iota \sigma(z_1) &= \sigma^{-1} \iota (z_1);\\
\iota (z_1\oplus_i z_2) &= (\iota z_1)\oplus_i (\iota z_2).
\end{split}
\end{align}
%
\begin{align}\label{eqn:r-coh}
  \text{\it coherence:}\quad\quad
\begin{split}
z_1 \oplus_0 z_2 \equiv z_1 \oplus_1 z_2 &\mod \{e_1,e_2\};\\
e(z_1\oplus_1 z_2) \equiv 0 &\mod \{e_1,e_2\}.
\end{split}
\end{align}
The first identity of (\ref{eqn:r-coh}) inverts $\delta_0\delta_1$,
and the second inverts $\delta_1$. Proofs of (\ref{eqn:r-coh}) use
polynomial division.

\subsection{projective curve and dichotomy}\label{sec:dichot}

Let $k$ be a field of characteristic different from two.  We let
$\Eaff$ be the set of zeros of Equation (\ref{eqn:t}) in $k^2$.  Let
$\Eoo\subset \Eaff$ be the subset of $\Eaff$ with nonzero coordinates
$x,y\ne0$.

We construct the projective Edwards curve $E$ by taking two copies of
$\Eaff$, glued along $\Eoo$ by isomorphism $\tau$.  We write $[P,i]\in
E$, with $i\in \ring{Z}/2\ring{Z}=\ring{F}_2$, for the image of $P\in
\Eaff$ in $E$ using the $i$th copy of $\Eaff$.  The gluing condition
gives for $P\in \Eoo$:
\begin{equation}\label{eqn:glue}
[P,i]=[\tau P,i+1].
\end{equation}

The group $G$ acts on the set $E$, specified on generators $\rho,\tau$
by $\rho[P,i]=[\rho(P),i]$ and $\tau[P,i]=[P,i+1]$.

We define addition on $E$ by
\begin{equation}\label{eqn:add-proj}
[P,i]\oplus [Q,j] = [P\oplus_\ell Q,i+j],\quad 
\text{if } \delta_\ell(P,Q)\ne 0,\quad \ell\in\ring{F}_2
\end{equation}
We will show that the addition is well-defined, is defined for all
pairs of points in $E$, and that it gives a group law with identity
element $[(1,0),0]$.  The inverse is $[P,i]\mapsto [\iota P,i]$, which
is well-defined by the inverse rules (\ref{eqn:r-iota}).

\begin{lemma} \label{lemma:no-fix} $G$ acts without fixed point on
  $\Eoo$.  That is, $g P = P$ implies that $g=1_G\in G$.
\end{lemma}

\begin{proof} Write $P=(x,y)$.  If $g = \rho^k\ne 1_G$, then $g P = P$
  implies that $2x=2y=0$ and $x=y=0$ (if the characteristic is not
  two), which is not a point on the curve.  If $g = \tau \rho^k$, then
  the fixed-point condition $g P = P$ leads to $2t x y=0$ or $t x^2 =
  t y^2 =\pm 1$.  Then $e(x,y) = 2 (\pm1-t)/t\ne0$, and again $P$ is
  not a point on the curve.\qed
\end{proof}

The domain of $\oplus_i$ is
\[
\Eaf{i} := \{(P,Q)\in \Eaff^2\mid \delta_i(P,Q)\ne0\}.
\]
Whenever we write $P\oplus_i Q$, it is always accompanied by the
implicit assertion of summability; that is, $(P,Q)\in \Eaf{i}$.

There is a group isomorphism $\ang{\rho}\to \Eaff\setminus\Eoo$ given by
\[
g\mapsto g(1,0)\in\{\pm (1,0),\pm (0,1)\} = \Eaff\setminus \Eoo.
\]

\begin{lemma}[dichotomy]\label{lemma:noco} 
Let $P,Q\in \Eaff$.  Then either $P\in \Eoo$ and $Q=g \iota\, P$ for
some $g\in \tau\ang{\rho}$, or $(P,Q)\in \Eaf{i}$ for some $i$.
Moreover, assume that $P\oplus_i Q = (1,0)$ for some $i$, then $Q =
\iota\,P$.
\end{lemma}

\begin{proof}  
We start with the first claim.
We analyze the denominators in the formulas for $\oplus_i$.  
We have $(P,Q)\in\Eaf{0}$ for all $P$ or $Q\in \Eaff\setminus\Eoo$.
That case completed,  we may assume that $P,Q\in \Eoo$.
  Assuming
  \[
  \delta_0(P,Q) = \delta_{0x}(P,Q)\delta_{0y}(P,Q)=0,\quad\text{and}\quad
  \delta_1(P,Q) = \delta_{1x}(P,Q)\delta_{1y}(P,Q)=0,
  \]
  we show that $Q = g \iota P$ for some $g\in \tau\ang{\rho}$.
  Replacing $Q$ by $\rho Q$ if needed, which exchanges
  $\delta_{0x}\leftrightarrow \delta_{0y}$, we may assume that
  $\delta_{0x}(P,Q)=0$.  Set $\tau Q = Q_0 = (a_0,b_0)$ and
  $P=(a_1,b_1)$.  

We claim that
\begin{equation}\label{eqn:gp-}
(a_0,b_0) \in \{\pm (b_1,a_1)\} \subset \Go\iota\,P.
\end{equation}
We describe
the main polynomial identity that must be verified.
Write $\delta',\delta_{+},\delta_{-}$ for $x_0 y_0\delta_{0x}$, $t x_0
y_0\delta_{1x}$, and $t x_0 y_0 \delta_{1y}$ respectively, each
evaluated at $(P,\tau(Q_0))=(x_1,y_1,1/(t x_0),1/(t y_0))$.  The
nonzero factors $x_0y_0$ and $t x_0 y_0$ have been included to clear
denominators, leaving us with polynomials.

We have two cases $\pm$, according to $\delta_{\pm}=0$.  In each case,
let
\[
S_\pm = \text{Gr\"obner basis of } \{e_1,e_2, 
\delta',\delta_{\pm} 
\}.
\]
We have
\begin{align}\label{eqn:dichot}
\begin{split}
(x_0^2-y_1^2, 
 y_0^2-x_1^2, 
 x_0 y_0 - x_1 y_1 
) &\equiv (0,0,0) \mod S_+\\
(2 x_0 y_0 (x_0^2-y_1^2), 
2 (1-t^2) x_0 y_0 (y_0^2-x_1^2), 
x_0 y_0 - x_1 y_1 
) &\equiv (0,0,0) \mod S_-.
\end{split}
\end{align}
In fact, $\delta' = x_0 y_0-x_1 y_1$, so that the ideal membership for
this polynomial is immediate.  The factors $2$, $1-t^2$, and $x_0 y_0$
are nonzero and can be removed from the left-hand side.
These equations then immediately yield $(a_0,b_0) = \pm (b_1,a_1)$.  
This gives the needed identity:  $\tau Q =
Q_0 = (a_0,b_0) = g \iota\,P$, for some $g\in \Go$.  Then $Q = \tau g
\iota\,P$.

The second statement of the lemma has a similar proof.  Polynomial
division gives for $i\in \ring{F}_2$:
\[
(x_1-x_2,y_1+y_2)\equiv (0,0) \mod \text{Gr\"obner} 
\{ e_1,e_2,q_x \delta_{ix}-1,q_y \delta_{iy}-1,
\nu_{i y},\nu_{i x}-\delta_{i x} \}.
\]
In fact, both $x_1-x_2$ and $y_1+y_2$ (which specify the condition $Q
=\iota\,P$) are already members of the Gr\"obner basis.  The fresh
variables $q_x,q_y$ force the denominators $\delta_{ix}$ and
$\delta_{iy}$ to be invertible.  Here the equations $\nu_{i y}=\nu_{i
  x}-\delta_{i x}=0$ specify the sum
$(1,0)=(\nu_{ix}/\delta_{ix},\nu_{iy}/\delta_{iy})$ of $Q$ and $P$.
\qed\end{proof}

\begin{lemma}[covering]\label{lemma:cov} 
  The rule (\ref{eqn:add-proj}) defining
  $\oplus$ assigns at least one value for every pair of points in $E$.
\end{lemma}

\begin{proof} If $Q=\tau \rho^k \iota\,P$, then $\tau Q$ does not have
  the form $\tau\rho^k\iota P$ because the action of $G$ is
  fixed-point free.  By dichotomy,
\begin{equation}\label{eqn:tt}
[P,i]\oplus [Q,j] = [P\oplus_\ell \tau Q,i+j+1]
\end{equation}
works for some $\ell$.  Otherwise, by dichotomy $P\oplus_\ell Q$ is
defined for some $\ell$.
\qed\end{proof}

\begin{lemma}[well-defined] Addition $\oplus$ given by
  (\ref{eqn:add-proj}) on $E$ is well-defined.
\end{lemma}

\begin{proof}
  The right-hand side of (\ref{eqn:add-proj}) is well-defined by
  coherence (\ref{eqn:r-coh}), provided we show well-definedness
  across gluings (\ref{eqn:glue}).  We use dichotomy.  If $Q=\tau
  \rho^k \iota\,P$, then by an easy simplification of polynomials,
\[
\delta_0(z,\tau\rho^k\iota z)=\delta_1(z,\tau\rho^k\iota z)=0.
\]
so that only one rule (\ref{eqn:tt}) for $\oplus$ applies (up to
coherence (\ref{eqn:r-coh}) and inversion (\ref{eqn:r-tau})), making
it necessarily well-defined.  Otherwise, coherence (\ref{eqn:r-coh}),
inversion (\ref{eqn:r-tau}), and (\ref{eqn:tauplus})) give when
$[Q,j]=[\tau Q,j+1]$:
 \[ 
[P\oplus_k \tau Q,i+j+1]=[\tau(P\oplus_k \tau Q),i+j] =
 [P\oplus_{k+1} Q,i+j] = [P\oplus_\ell Q,i+j].
\]
\qed\end{proof}

\subsection{group law}

\begin{theorem}\label{thm:proj-group}  $E$ is an abelian group.
\end{theorem}

\begin{proof} We have already shown the existence of an identity and
  inverse.

  We prove associativity.  Both sides of the associativity identity
  are clearly invariant under shifts $[P,i]\mapsto [P,i+j]$ of the
  indices.  Thus, it is enough to show
\[
[P,0] \oplus ([Q,0]\oplus [R,0]) = ([P,0]\oplus [Q,0])\oplus [R,0].
\]
By polynomial division, we have the following associativity identities
\begin{equation}\label{eqn:assoc-affine}
 (z_1\oplus_k z_2)\oplus_\ell z_3 \equiv z_1 
\oplus_i (z_2\oplus_j z_3) \mod \{e_1,e_2,e_3\}
\end{equation}
in the appropriate localizations, for $i,j,k,\ell\in \ring{F}_2$.

Note that $(g [P_1,i])\oplus [P_2,j] = g([P_1,i]\oplus [P_2,j])$ for
$g\in G$, as can easily be checked on generators $g=\tau,\rho$ of $G$,
using dichotomy, (\ref{eqn:add-proj}), and (\ref{eqn:r-rho}).  We use
this to cancel group elements $g$ from both sides of equations without
further comment.

We claim that
\begin{equation}\label{eqn:semi}
([P,0]\oplus [Q,0])\oplus [\iota\,Q,0] = [P,0].
\end{equation}
The special case $Q= \tau\rho^k \iota(P)$ is easy. We reduce the claim
to the case where $P\oplus_\ell Q\ne \tau\rho^k Q$, by applying $\tau$
to both sides of (\ref{eqn:semi}) and replacing $P$ with $\tau P$ if
necessary.  Then by dichotomy, the left-hand side simplifies by affine
associativity \ref{eqn:assoc-affine} to give the claim.

Finally, we have general associativity by repeated use of dichotomy,
which reduces in each case to (\ref{eqn:assoc-affine}) or
(\ref{eqn:semi}).
\qed\end{proof}

\subsection{formalization in Isabelle/HOL of projective Edwards curves}

Following the change of variables performed in Section~\ref{sec:defs}, 
it is assumed that $c = 1$ and $d = t^2$ where
$t \neq -1,0,1$. The resulting formalization is more challenging. In the
following, some key insights are emphasized.

\subsubsection{Gr\"{o}bner basis}

The proof of Lemma~\ref{lemma:noco} (dichotomy) requires solving
particular instances of the ideal membership problem.  Formalization
caught and corrected some ideal membership errors in
\cite{hales2016group}, which resulted from an incorrect interpretation
of computer algebra calculations.  For instance, a goal
\[
\exists r_1 \, r_2 \, r_3 \, r_4.\ 
y_0^2 - x_1^2 = r_1 e(x_0,y_0) + r_2 e(x_1,y_1) + r_3 \delta' + r_4
\delta_{-} 
\] 
(derived from \cite{hales2016group})
had to be corrected to 
\[
\exists r_1 \, r_2 \, r_3 \, r_4.\ 
2 x_0 y_0 (y_0^2 - x_1^2) = r_1 e(x_0,y_0) + r_2 e(x_1,y_1) + r_3 \delta' + r_4
\delta_{-}
\]
to prove (\ref{eqn:dichot}).
In another subcase, it was necessary to strengthen the hypothesis
$\delta_{+} = 0$ to $\delta_{-} \neq 0$. Eventually, after some reworking,
\textit{algebra} solved the required ideal membership
problems.

\subsubsection{definition of the group addition}

We defined the addition in three stages. This is convenient for some
lemmas like covering (Lemma~\ref{lemma:cov}). 
First, we define the addition on projective
points (Figure~\ref{fig3}). Then, we add two classes of points by
applying the basic addition to any pair of points coming from each
class. Finally, we apply the gluing relation and obtain as a result a
set of classes with a unique element, which is then defined as the
resulting class (Figure~\ref{fig4}).

\begin{figure}
{\input{proj-add-1.tex}}
\caption{Definition of $\oplus$ on points}
\label{fig3}
\end{figure}
\begin{figure}
	{\input{proj-add-2.tex}}
	{\input{proj-add-3.tex}}
	\caption{Definition of $\oplus$ on classes}
	\label{fig4}
\end{figure}

The definitions use Isabelle's ability to encode partial
functions. However, it is possible to obtain an equivalent definition
more suitable for execution. In particular, it is easy to compute the
gluing relation (see lemmas $\text{e\_proj\_elim\_1}$,
$\text{e\_proj\_elim\_2}$ and $\text{e\_proj\_aff}$ in the
formalization scripts).

Finally, since projective addition works with classes, we had to show
that its definition does not depend on the representative used.

\subsubsection{proof of associativity}

\begin{table}
\begin{center}
\begin{tabular}{ r c l }
  $\delta \; \tau P_1 \; \tau P_2 \neq 0$
  & $\implies$ & $\; \delta \; P_1 \; P_2 \neq 0$ \\ 
  $\delta' \; \tau P_1 \; \tau P_2 \neq 0$
  & $\implies$ & $\; \delta' \; P_1 \; P_2 \neq 0$ \\ 
  $\delta \; P_1 \; P_2 \neq 0$,\quad $\; \delta \; P_1 \; \tau P_2 \neq 0$
  & $\implies$ & $\delta' \; P_1 \; P_2 \neq 0$ \\ 
  $\delta' \; P_1 \; P_2 \neq 0$,\quad $\; \delta' \; P_1 \; \tau P_2 \neq 0$
  & $\implies$ & $\delta \; P_1 \;P_2 \neq 0$ \\
  \hline
  $\delta' \; (P_1 \oplus_1 P_2) \; \tau \iota P_2 \neq 0$
  & $\implies$ & $\; \delta \; (P_1 \oplus_1 P_2) \; \iota P_2 \neq 0$ \\ 
  $\delta \; P_1 \; P_2 \neq 0$,\quad $\; \delta \; (P_1 \oplus_0 P_2) \; \tau \iota P_2 \neq 0$
  & $\implies$ & $\; \delta' (P_1 \oplus_0 P_2) \; \iota P_2 \neq 0$ \\ 
  $\delta \; P_1 \; P_2 \neq 0$,\quad $\; \delta' \; (P_0 \oplus_0 P_1) \; \tau \iota P_2 \neq 0$
  & $\implies$ & $\; \delta \; (P_0 \oplus_0 P_1) \; \iota P_2 \neq 0$ \\ 	    
  $\delta' \; P_1 \; P_2 \neq 0$,\quad $\; \delta \; (P_0 \oplus_1 P_1) \; \tau \iota P_2 \neq 0$
  & $\implies$ & $\; \delta' \; (P_0 \oplus_1 P_1) \; \iota P_2 \neq 0$ \\
\end{tabular}
\end{center}
	\caption{List of $\delta$ relations}
	\label{table:1}
\end{table}

During formalization, we found several relations between $\delta$
expressions (see Table~\ref{table:1}). While they were proven in order
to show associativity, the upper group can rather be used to establish
the independence of class representative and the lower group is
crucial to establish the associativity law.

In particular, the lower part of the table is fundamental to the formal
proof of
Equation~(\ref{eqn:semi}).
In more detail, the formal proof development showed that it was
necessary to perform a dichotomy (Lemma~\ref{lemma:noco}) three
times. The first dichotomy is performed on $P$, $Q$. The non-summable
case was easy. Therefore, we set $R = P \oplus Q$. On each of the
resulting branches, a dichotomy on $R$, $\iota Q$ is performed. This
time the summable cases were easy, but the non-summable case required a
third dichotomy on $R,\tau \iota Q$. The non-summable case was solved
using the no-fixed-point theorem but for the summable subcases the
following expression is obtained:
\[
([P,0] \oplus [Q,0]) \oplus [\tau \iota Q,0] =
[(P \oplus Q) \oplus \tau \iota Q,0]
\] 
Here we cannot invoke associativity because $Q$, $\tau \iota Q$ are
non-summable (lemma $\text{not\_add\_self}$).  Instead, we use the
equations from the lower part of the table and the hypothesis of the
second dichotomy to get a contradiction.

\section{Conclusion}

We have shown that Isabelle can encompass the process
of defining, computing and certifying intensive algebraic
calculations. The encoding in a proof-assistant allows a better
comprehension of the methods used and helps to clarify its structure.

\newpage


\bibliography{refs} 
\bibliographystyle{alpha}

\end{document}

%% file: proj-add-0.tex
\begin{isabellebody}%

add\ {\isacharcolon}{\isacharcolon}\ {\isacharprime}a\ {\isasymtimes}\ {\isacharprime}a\ {\isasymRightarrow}\ {\isacharprime}a\ {\isasymtimes}\ {\isacharprime}a\ {\isasymRightarrow}\ {\isacharprime}a\ {\isasymtimes}\ {\isacharprime}a\ \isanewline
add\ {\isacharparenleft}x\isactrlsub {\isadigit{1}}{\isacharcomma}y\isactrlsub {\isadigit{1}}{\isacharparenright}\ {\isacharparenleft}x\isactrlsub {\isadigit{2}}{\isacharcomma}y\isactrlsub {\isadigit{2}}{\isacharparenright}\ {\isacharequal}
{\isacharparenleft}{\isacharparenleft}x\isactrlsub {\isadigit{1}}{\isacharasterisk}x\isactrlsub {\isadigit{2}}\ {\isacharminus}\ c{\isacharasterisk}y\isactrlsub {\isadigit{1}}{\isacharasterisk}y\isactrlsub {\isadigit{2}}{\isacharparenright}\ div {\isacharparenleft}{\isadigit{1}}{\isacharminus}d{\isacharasterisk}x\isactrlsub {\isadigit{1}}{\isacharasterisk}y\isactrlsub {\isadigit{1}}{\isacharasterisk}x\isactrlsub {\isadigit{2}}{\isacharasterisk}y\isactrlsub {\isadigit{2}}{\isacharparenright}{\isacharcomma}\ \isanewline
\ \ \ \ \ \ \ \ \ \ \ \ \ \ \ \ \ \ \ \ \ \ \ {\isacharparenleft}x\isactrlsub {\isadigit{1}}{\isacharasterisk}y\isactrlsub {\isadigit{2}}{\isacharplus}y\isactrlsub {\isadigit{1}}{\isacharasterisk}x\isactrlsub {\isadigit{2}}{\isacharparenright}\ div\ {\isacharparenleft}{\isadigit{1}}{\isacharplus}d{\isacharasterisk}x\isactrlsub {\isadigit{1}}{\isacharasterisk}y\isactrlsub {\isadigit{1}}{\isacharasterisk}x\isactrlsub {\isadigit{2}}{\isacharasterisk}y\isactrlsub {\isadigit{2}}{\isacharparenright}{\isacharparenright}
\end{isabellebody}%

%% file: proj-add-4.tex
\begin{isabellebody}%

\ \ \isacommand{have}\isamarkupfalse%
\ {\isachardoublequoteopen}{\isasymexists}\ r{\isadigit{1}}\ r{\isadigit{2}}\ r{\isadigit{3}}{\isachardot}\ gxpoly\ {\isacharequal}\ r{\isadigit{1}}\ {\isacharasterisk}\ e{\isadigit{1}}\ {\isacharplus}\ r{\isadigit{2}}\ {\isacharasterisk}\ e{\isadigit{2}}\ {\isacharplus}\ r{\isadigit{3}}\ {\isacharasterisk}\ e{\isadigit{3}}{\isachardoublequoteclose}\isanewline
\ \ \ \ \isacommand{unfolding}\isamarkupfalse%
\ gxpoly{\isacharunderscore}def\ g\isactrlsub x{\isacharunderscore}def\ Delta\isactrlsub x{\isacharunderscore}def\ \isanewline
\ \ \ \ \isacommand{apply}\isamarkupfalse%
{\isacharparenleft}simp\ add{\isacharcolon}\ assms{\isacharparenleft}{\isadigit{1}}{\isacharcomma}{\isadigit{2}}{\isacharparenright}{\isacharparenright}\isanewline
\ \ \ \ \isacommand{apply}\isamarkupfalse%
{\isacharparenleft}rewrite\ \isakeyword{in}\ {\isachardoublequoteopen}{\isacharunderscore}\ {\isacharslash}\ {\isasymhole}{\isachardoublequoteclose}\ delta{\isacharunderscore}minus{\isacharunderscore}def{\isacharbrackleft}symmetric{\isacharbrackright}{\isacharparenright}{\isacharplus}\isanewline
\ \ \ \ \isacommand{apply}\isamarkupfalse%
{\isacharparenleft}simp\ add{\isacharcolon}\ divide{\isacharunderscore}simps\ assms{\isacharparenleft}{\isadigit{9}}{\isacharcomma}{\isadigit{1}}{\isadigit{1}}{\isacharparenright}{\isacharparenright}\isanewline
\ \ \ \ \isacommand{apply}\isamarkupfalse%
{\isacharparenleft}rewrite\ left{\isacharunderscore}diff{\isacharunderscore}distrib{\isacharparenright}\isanewline
\ \ \ \ \isacommand{apply}\isamarkupfalse%
{\isacharparenleft}simp\ add{\isacharcolon}\ simp{\isadigit{1}}gx\ simp{\isadigit{2}}gx{\isacharparenright}\isanewline
\ \ \ \ \isacommand{unfolding}\isamarkupfalse%
\ delta{\isacharunderscore}plus{\isacharunderscore}def\ delta{\isacharunderscore}minus{\isacharunderscore}def\isanewline
\ \ \ \ \ \ \ \ \ \ \ \ \ \ e{\isadigit{1}}{\isacharunderscore}def\ e{\isadigit{2}}{\isacharunderscore}def\ e{\isadigit{3}}{\isacharunderscore}def\ e{\isacharunderscore}def\isanewline
\ \ \ \ \isacommand{by}\isamarkupfalse%
\ algebra\isanewline

\end{isabellebody}%

%% file: proj-add-1.tex
\begin{isabellebody}%

\isacommand{type{\isacharunderscore}synonym}\isamarkupfalse%
\ {\isacharparenleft}{\isacharprime}b{\isacharparenright}\ ppoint\ {\isacharequal}\ {\isacartoucheopen}{\isacharparenleft}{\isacharparenleft}{\isacharprime}b\ {\isasymtimes}\ {\isacharprime}b{\isacharparenright}\ {\isasymtimes}\ bit{\isacharparenright}{\isacartoucheclose}\isanewline

p{\isacharunderscore}add\ {\isacharcolon}{\isacharcolon}\ {\isacharprime}a\ ppoint\ {\isasymRightarrow}\ {\isacharprime}a\ ppoint\ {\isasymRightarrow}\ {\isacharprime}a\ ppoint\ \isakeyword{where}\ \isanewline
\ \ p{\isacharunderscore}add\ {\isacharparenleft}{\isacharparenleft}x\isactrlsub {\isadigit{1}}{\isacharcomma}\ y\isactrlsub {\isadigit{1}}{\isacharparenright}{\isacharcomma}\ l{\isacharparenright}\ {\isacharparenleft}{\isacharparenleft}x\isactrlsub {\isadigit{2}}{\isacharcomma}\ y\isactrlsub {\isadigit{2}}{\isacharparenright}{\isacharcomma}\ j{\isacharparenright}\ {\isacharequal}\ {\isacharparenleft}add\ {\isacharparenleft}x\isactrlsub {\isadigit{1}}{\isacharcomma}\ y\isactrlsub {\isadigit{1}}{\isacharparenright}\ {\isacharparenleft}x\isactrlsub {\isadigit{2}}{\isacharcomma}\ y\isactrlsub {\isadigit{2}}{\isacharparenright}{\isacharcomma}\ l{\isacharplus}j{\isacharparenright}\isanewline
\ \isakeyword{if}\ delta\ x\isactrlsub {\isadigit{1}}\ y\isactrlsub {\isadigit{1}}\ x\isactrlsub {\isadigit{2}}\ y\isactrlsub {\isadigit{2}}\ {\isasymnoteq}\ {\isadigit{0}}\ {\isasymand}\ {\isacharparenleft}x\isactrlsub {\isadigit{1}}{\isacharcomma}\ y\isactrlsub {\isadigit{1}}{\isacharparenright}\ {\isasymin}\ e{\isacharprime}{\isacharunderscore}aff\ {\isasymand}\ {\isacharparenleft}x\isactrlsub {\isadigit{2}}{\isacharcomma}\ y\isactrlsub {\isadigit{2}}{\isacharparenright}\ {\isasymin}\ e{\isacharprime}{\isacharunderscore}aff\ \isanewline
{\isacharbar}\ p{\isacharunderscore}add\ {\isacharparenleft}{\isacharparenleft}x\isactrlsub {\isadigit{1}}{\isacharcomma}\ y\isactrlsub {\isadigit{1}}{\isacharparenright}{\isacharcomma}\ l{\isacharparenright}\ {\isacharparenleft}{\isacharparenleft}x\isactrlsub {\isadigit{2}}{\isacharcomma}\ y\isactrlsub {\isadigit{2}}{\isacharparenright}{\isacharcomma}\ j{\isacharparenright}\ {\isacharequal}\ {\isacharparenleft}ext{\isacharunderscore}add\ {\isacharparenleft}x\isactrlsub {\isadigit{1}}{\isacharcomma}\ y\isactrlsub {\isadigit{1}}{\isacharparenright}\ {\isacharparenleft}x\isactrlsub {\isadigit{2}}{\isacharcomma}\ y\isactrlsub {\isadigit{2}}{\isacharparenright}{\isacharcomma}\ l{\isacharplus}j{\isacharparenright}\isanewline
\ \isakeyword{if}\ delta{\isacharprime}\ x\isactrlsub {\isadigit{1}}\ y\isactrlsub {\isadigit{1}}\ x\isactrlsub {\isadigit{2}}\ y\isactrlsub {\isadigit{2}}\ {\isasymnoteq}\ {\isadigit{0}}\ {\isasymand}\ {\isacharparenleft}x\isactrlsub {\isadigit{1}}{\isacharcomma}\ y\isactrlsub {\isadigit{1}}{\isacharparenright}\ {\isasymin}\ e{\isacharprime}{\isacharunderscore}aff\ {\isasymand}\ {\isacharparenleft}x\isactrlsub {\isadigit{2}}{\isacharcomma}\ y\isactrlsub {\isadigit{2}}{\isacharparenright}\ {\isasymin}\ e{\isacharprime}{\isacharunderscore}aff

\end{isabellebody}%

%% file: proj-add-2.tex
\begin{isabellebody}%
\isacommand{type{\isacharunderscore}synonym}\isamarkupfalse%
\ {\isacharparenleft}{\isacharprime}b{\isacharparenright}\ pclass\ {\isacharequal}\ {\isacartoucheopen}{\isacharparenleft}{\isacharprime}b{\isacharparenright}\ ppoint\ set{\isacartoucheclose} \\

proj{\isacharunderscore}add{\isacharunderscore}class\ {\isacharcolon}{\isacharcolon}\ {\isacharparenleft}{\isacharprime}a{\isacharparenright}\ pclass\ {\isasymRightarrow}\ {\isacharparenleft}{\isacharprime}a{\isacharparenright}\ pclass\ {\isasymRightarrow}\ {\isacharparenleft}{\isacharprime}a{\isacharparenright}\ pclass\ set\isanewline
\ \ proj{\isacharunderscore}add{\isacharunderscore}class\ c\isactrlsub {\isadigit{1}}\ c\isactrlsub {\isadigit{2}}\ {\isacharequal}\ \ \ \ \ \ \ \isanewline
\ \ \ \ \ \ \ \ {\isacharparenleft}p{\isacharunderscore}add\ {\isacharbackquote}\ {\isacharbraceleft}{\isacharparenleft}{\isacharparenleft}{\isacharparenleft}x\isactrlsub {\isadigit{1}}{\isacharcomma}\ y\isactrlsub {\isadigit{1}}{\isacharparenright}{\isacharcomma}\ i{\isacharparenright}{\isacharcomma}{\isacharparenleft}{\isacharparenleft}x\isactrlsub {\isadigit{2}}{\isacharcomma}\ y\isactrlsub {\isadigit{2}}{\isacharparenright}{\isacharcomma}\ j{\isacharparenright}{\isacharparenright}{\isachardot} \ \isanewline
\ \ \ \ \ \ \ \ \ \ \ \ \ \ \ \ {\isacharparenleft}{\isacharparenleft}{\isacharparenleft}x\isactrlsub {\isadigit{1}}{\isacharcomma}\ y\isactrlsub {\isadigit{1}}{\isacharparenright}{\isacharcomma}\ i{\isacharparenright}{\isacharcomma}{\isacharparenleft}{\isacharparenleft}x\isactrlsub {\isadigit{2}}{\isacharcomma}\ y\isactrlsub {\isadigit{2}}{\isacharparenright}{\isacharcomma}\ j{\isacharparenright}{\isacharparenright}\ {\isasymin}\ c\isactrlsub {\isadigit{1}}\ {\isasymtimes}\ c\isactrlsub {\isadigit{2}}\ {\isasymand}\ \isanewline
\ \ \ \ \ \ \ \ \ \ \ \ \ \ \ \ {\isacharparenleft}{\isacharparenleft}x\isactrlsub {\isadigit{1}}{\isacharcomma}\ y\isactrlsub {\isadigit{1}}{\isacharparenright}{\isacharcomma}\ {\isacharparenleft}x\isactrlsub {\isadigit{2}}{\isacharcomma}\ y\isactrlsub {\isadigit{2}}{\isacharparenright}{\isacharparenright}\ {\isasymin}\ e{\isacharprime}{\isacharunderscore}aff{\isacharunderscore}{\isadigit{0}}\ {\isasymunion}\ e{\isacharprime}{\isacharunderscore}aff{\isacharunderscore}{\isadigit{1}}{\isacharbraceright}{\isacharparenright}\ {\isacharslash}{\isacharslash}\ gluing\ \isanewline 
\ \isakeyword{if}\ c\isactrlsub {\isadigit{1}}\ {\isasymin}\ e{\isacharunderscore}proj\ \isakeyword{and}\ c\isactrlsub {\isadigit{2}}\ {\isasymin}\ e{\isacharunderscore}proj\ 
\isanewline

\end{isabellebody}%

%% file: proj-add-3.tex
\begin{isabellebody}%

proj{\isacharunderscore}addition\ c\isactrlsub {\isadigit{1}}\ c\isactrlsub {\isadigit{2}}\ {\isacharequal}\ the{\isacharunderscore}elem\ {\isacharparenleft}proj{\isacharunderscore}add{\isacharunderscore}class\ c\isactrlsub {\isadigit{1}}\ c\isactrlsub {\isadigit{2}}{\isacharparenright}

\end{isabellebody}%